\documentclass[11pt,a4paper]{article}

\usepackage[hmargin={25mm,25mm},vmargin={25mm,25mm}]{geometry}
\usepackage[utf8]{inputenc}
\usepackage{latexsym}
\usepackage{amssymb}
\usepackage{amsmath}
\usepackage{amsfonts}
\usepackage{mathtools}
\usepackage{commath}
\usepackage{amsthm}
\usepackage{bbm}
\usepackage{csquotes}
\usepackage{theoremref}
\usepackage{gensymb}

\usepackage{amsthm, mathtools, graphicx, paralist, cite, color,verbatim}
\usepackage{comment}
\usepackage{hyperref}

\usepackage{amsmath, nccmath} 
\usepackage{amsthm}
\usepackage{indentfirst}
\usepackage{appendix}
\usepackage{listings}
\usepackage{hyphenat}
\usepackage{physics}
\usepackage{calrsfs}
\usepackage{enumerate}
\usepackage{tikz-cd}
\usepackage{mathrsfs}

\theoremstyle{plain}
\newtheorem{theorem}{Theorem}[section]
\newtheorem{lemma}[theorem]{Lemma}

\theoremstyle{definition}

\begin{document}

 \title{A note on connectivity in directed graphs}
 
 \author{Stelios Stylianou\footnote{School of Mathematics, University of Bristol. Supported by an EPSRC Doctoral Training Studentship.}}
\date{September 2024}
\maketitle

\begin{abstract}
    We say a directed graph $G$ on $n$ vertices is {\em irredundant} if the removal of any edge reduces the number of ordered pairs of distinct vertices $(u,v)$ such that there exists a directed path from $u$ to $v$. We determine the maximum possible number of edges such a graph can have, for every $n \in \mathbb{N}$. We also characterize the cases of equality. This resolves, in a strong form, a question of Crane and Russell.
\end{abstract}

\section{Introduction}

In this paper, a {\em directed graph} $G$ is a pair of finite sets $(V,E)$ such that $E \subset \{(u,v):\ u,v \in V,\ u \neq v\}$; as usual, the set $V = V(G)$ is called the {\em vertex-set} of $G$, and the set $E = E(G)$ is called the {\em edge-set} of $G$. (So, we do not allow loops at any vertex, or multiple edges going between two vertices in the same direction, but we do allow two edges between the same pair of vertices in different directions.) We write $e(G):=\abs{E(G)}$ for the number of edges of $G$, also known as the {\em size} of $G$.

As usual, if $(u,v) \in E(G)$ then we say there is a {\em directed edge from $u$ to $v$}. We will also write $u \xrightarrow{G} v$ if there is any directed path from $u$ to $v$. When it is clear what our graph is from the context, we will sometimes simply write $V$ for the vertex set, $E$ for the edge set, and $u \rightarrow v$ to indicate that there is a directed path from $u$ to $v$. We will sometimes refer to a directed path simply as a {\em path}, for brevity.

We say that a directed graph $G$ is \emph{irredundant} if the following holds: for any edge $e \in E(G)$, there exist $u,v \in V(G)$ such that $u \xrightarrow{G} v$ holds, but not $u \xrightarrow{G \setminus e} v$. Note that this is equivalent to saying that, for any $(x,y) \in E(G)$, there is no other path from $x$ to $y$, except for the edge $(x,y)$. (Indeed, if the latter holds, then $G$ is clearly irredundant. Conversely, suppose $G$ is irredundant and let $e=(x,y) \in E(G)$. Fix $u,v \in V$ such that $u \xrightarrow{G} v$, but not $u \xrightarrow{G \setminus e} v$. Then the path from $u$ to $v$ in $G$ certainly includes $e$, so it is of the form $u\ldots xy \ldots v$. Let $P_1$ and $P_2$ be the paths used to go from $u$ to $x$ and from $y$ to $v$ respectively. If there exists a path $P$ from $x$ to $y$ that does not include $e$, then it is possible to go from $u$ to $v$ in $G \setminus e$ by concatenating $P_1$, $P$ and $P_2$, a contradiction.)

We say that an edge $e=(u,v)$ is \emph{bad} if $u \xrightarrow{G \setminus e} v$, so $G$ being irredundant is equivalent to $G$ having no bad edges. For $n \in \mathbb{N}$, we write $f(n)$ for the maximum number of edges in an irredundant graph on $n$ vertices. Note that $f(1)=0$ and $f(2)=2$. We say that a graph $G$ on $n$ vertices is \emph{maximum irredundant} if $e(G)=f(n)$.

Crane and Russell (personal communication) raised the natural question of determining, for each $n \in \mathbb{N}$, the maximum possible number of edges in an irredundant directed graph on $n$ vertices. Our purpose in this paper is to answer this question, and also to characterize the extremal graphs. We prove the following.

\begin{theorem}
\label{thm:main}

For any $n \in \mathbb{N}$, the number of edges $f(n)$ in a maximum irredundant graph on $n$ vertices satisfies
$$f(n)= \begin{cases*}
2n-2 & for $n\leq 7$, \\
\lfloor n^2/4 \rfloor & for $n \geq 7.$    
\end{cases*}$$

For $n<7$, $G$ is maximum irredudant if and only if $G$ is a \enquote{double tree}, that is, a tree where all the edges are double. For $n>7$, $G$ is maximum irredundant if and only if $G$ is a  simple complete bipartite graph with part-sizes differing by at most one. (Here, \emph{simple complete bipartite} means that $G$ has no double edges and we can split its vertices into two parts, $U$ and $V$, such that $(x,y) \in E(G)$ if and only if $x \in U$ and $y \in V$.) For $n=7$, $G$ is maximum irredudant if and only if $G$ is either a double tree or a simple complete bipartite graph with part-sizes differing by at most one.
\end{theorem}

Our proof is a simple, combinatorial one, using induction on the number of vertices to get upper bounds on $f(n)$. We will get various bounds depending on certain features of our graph.

The main idea is to find a directed cycle $C$ of length two or three in $G$ (that is, either a double edge or a directed triangle) and then contract the vertices of this cycle into a single vertex $x$ to obtain a new graph $G'$. The vertex $x$ essentially plays the same role in $G'$ as the cycle $C$ used to play in $G$. More precisely, we observe that $G$ being irredundant implies that, for any vertex $y \in G \setminus C$, at most one edge can exist from $y$ to $C$ and at most one edge can exist from $C$ to $y$. This will allow us to add the corresponding edge from $y$ to $x$ or from $x$ to $y$ in our new graph, without affecting the total number of edges (apart from the few edges removed when contracting $C$ into $x$).

Depending on whether $G$ contains a double edge or a directed triangle, we get upper bounds on $f(n)$ depending on $f(n-1)$ and $f(n-2)$ respectively (as we reduce the number of vertices by one and two respectively), which give linear bounds. In the remaining cases, we use Mantel's theorem \cite{Mantel} to deduce that $f(n) \leq \lfloor n^2/4 \rfloor$.


\section{Determining the maximal number of edges in an irredundant graph.}

We will use the following classical theorem of Mantel \cite{Mantel} for undirected graphs. (As usual, an {\em undirected graph} is a pair of finite sets $(V,E)$ where $E \subset \{\{u,v\}:\ u,v \in V,\ u \neq v\}$, i.e., $E$ is a set of unordered pairs of elements of $V$; the set $V$ is referred to as the {\em vertex-set}, and $E$ is referred to as the {\em edge-set}.)

\begin{theorem}
Let $n \in \mathbb{N}$, and let $G$ be a triangle-free, undirected graph on $n$ vertices. Then $e(G) \leq \lfloor n^2/4 \rfloor$. Equality holds if and only if $G$ is a complete bipartite graph where the two parts have sizes that differ by at most one.    
\end{theorem}

We first obtain upper bounds on $f(n)$ in three different cases:
\begin{itemize}
    \item $G$ contains some double edge.
    \item $G$ contains no double edge, but contains some triangle.
    \item $G$ contains neither a double edge nor a triangle.
\end{itemize}
Mantel's theorem will be useful in the third case.

\begin{lemma}
\label{main_lemma}
For any $n \geq 3$, $$f(n) \leq \max \{f(n-1)+2, f(n-2)+3, \lfloor n^2/4 \rfloor\}.$$
\end{lemma}

\begin{proof}
Let $G$ be an irredundant graph on $n$ vertices. Suppose first that $G$ contains some double edge, that is there exist $u,v \in V$ such that $(u,v), (v,u) \in E$. Observe that, for any other vertex $w$, we cannot have both $(w,u) \in E$ and $(w,v) \in E$. Indeed, in that case $wuv$ is a path and thus $(w,v)$ is a bad edge. Similarly, we cannot have both $(u,w) \in E$ and $(v,w) \in E$. We now remove vertices $u$ and $v$ to get a graph $H$. Then we add a new vertex $x$ and for any $y \in V(H)$ we add the edge $(x,y)$ if and only if either $(u,y)$ or $(v,y)$ was an edge in $G$. Similarly, we add the edge $(y,x)$ if and only if either $(y,u)$ or $(y,v)$ was an edge in $G$. Let $G'$ be the resulting graph. Note that $e(G')=e(G)-2$.

We show that $G'$ is also irredundant. Suppose $(a,b)$ is bad in $G'$. Fix some path $P$ from $a$ to $b$ in $G' \setminus (a,b)$. We will show that $(a,b)$ is bad in $G$. This is trivially true if $P$ does not pass through $x$. If $P$ passes through $x$, then we have paths $P_1$ from $a$ to $x$ and $P_2$ from $x$ to $b$. By the definition of $x$, this means that in $G \setminus (a,b)$ there exist paths $P_1'$ from $a$ to $c_1$ and $P_2'$ from $c_2$ to $b$, where each $c_i$ is equal to either $u$ or $v$. The existence of these two paths, together with the fact that $uv$ is a double edge, shows that there is a path from $a$ to $b$ in $G \setminus (a,b)$. Therefore, we have $e(G) = e(G')+2 \leq f(n-1)+2$.

We may therefore assume, henceforth, that $G$ contains no double edge. Suppose $G$ contains some (not necessarily directed) triangle on vertices $u_1, u_2$ and $u_3$. Without loss of generality, $(u_1,u_2), (u_2,u_3) \in E$. Since $G$ is irredundant, we cannot have $(u_1,u_3) \in E$. Therefore, $(u_3,u_1) \in E$. Observe that, for any other vertex $v$, we cannot have both $(u_i,v) \in E$ and $(u_j,v) \in E$ for distinct $i,j$. Similarly, we cannot have both $(v,u_i) \in E$ and $(v,u_j) \in E$ for distinct $i,j$. We now remove vertices $u_1, u_2$ and $u_3$ to get a graph $H$. Then we add a vertex $x$ and for any $y \in V(H)$ we add the edge $(x,y)$ if and only if $(u_i,y)$ was an edge in $G$ for some $i$. Similarly, we add the edge $(y,x)$ if and only if $(y,u_i)$ was an edge in $G$ for some $i$. Let $G'$ be the resulting graph. Note that $e(G')=e(G)-3$.

We show that $G'$ is also irredundant. Suppose $(a,b)$ is bad in $G'$. Similar to the case where $G$ was allowed to have double edges, we deduce that in $G \setminus (a,b)$ there exist paths $P_1'$ from $a$ to $c_1$ and $P_2'$ from $c_2$ to $b$, where each $c_i$ is equal to $u_j$ for some $j$. The existence of these two paths, together with the fact that $u_1,u_2,u_3$ form a directed cycle, shows that there is a path from $a$ to $b$ in $G \setminus (a,b)$. We have $e(G) = e(G')+3 \leq f(n-2)+3$.

Finally, we might assume that $G$ contains no triangles. Then $e(G) \leq \lfloor n^2/4 \rfloor$ by Mantel's theorem.
\end{proof}

Before proving our main result, we first prove a lemma that will help us characterize the maximum-sized irredundant graphs on at least seven vertices.

\begin{lemma}
\label{lemma_simple}
    Let $G$ be an irredundant complete bipartite graph on $n \geq 7$ verices. Then $G$ is simple.
\end{lemma}

\begin{proof}

We will show that $G$ cannot contain any directed path of length three. Fix vertices $1,2,3 \in U$ and $a,b,c \in V$ and assume that $(1,a), (a,2), (2,b) \in E$. If $(1,b) \in E$, then that would be a bad edge because of the path $1a2b$, so $(b,1) \in E$. Now assume that $(3,b) \in E$. Then we have the path $3b1a$, so $(a,3) \in E$. If $(c,2) \in E$, then we have the path $c2b1$, so $(1,c) \in E$. But then $(a,2)$ is bad, since we have the path $a3b1c2$. So suppose that $(2,c) \in E$. Then we have the path $1a2c$, so $(c,1) \in E$. But then $(2,b)$ is bad, since we have the path $2c1a3b$. Therefore, we must have $(b,3) \in E$, and by symmetry $(a,3) \in E$ as well.

Now assume that $(3,c) \in E$. We have the path $1a3c$, so $(c,1) \in E$. But then $(b,1)$ is bad, since we have the path $b3c1$. Therefore, we have $(c,3) \in E$. If $(c,1) \in E$, then $(c,3)$ is bad, since we have the path $c1a3$. Therefore, $(1,c) \in E$. If $(2,c) \in E$, then $1a2c$ is a path and thus $(1,c)$ is bad. Therefore, $(c,2) \in E$. But then $(c,3)$ is bad, since we have the path $c2b3$.

We have shown that $G$ has no paths of length three. If there are no paths of length two, then either every edge has direction from $U$ to $V$ or every edge has direction from $V$ to $U$. Now suppose there is a path of length two, say $(1,x), (x,2) \in E$, where $1,2 \in U$ and $x \in V$. Since there is no path of length three, we have $(1,v),(v,2) \in E$ for every $v \in V$. Now let $3 \in U$ be any other vertex. Suppose there exist $y,z \in V$ such that $(3,y),(z,3) \in E$. Then $z3y2$ is a path of length three, a contradiction. Therefore, for any $u \in U$, either $(u,v) \in E$ for all $v \in V$ or $(v,u) \in E$ for all $v \in V$.
\end{proof}

\begin{proof}[Proof of Theorem \ref{thm:main}.]
Note that $\lfloor n^2/4 \rfloor \leq 2n-2$ for $1 \leq n \leq 7$. The statement of the theorem is clearly true in the cases $n=1,2$. Let $3 \leq n \leq 7$ and suppose $f(n-2)=2n-6$ and $f(n-1)=2n-4$. Then, using Lemma \ref{main_lemma}, we deduce that $f(n) \leq 2n-2$. Now observe that any tree where all the edges are double is irredundant and has exactly $2n-2$ edges. Therefore, $f(n)=2n-2$ for $n \leq 7$.

We now observe that $f(n) \geq \lfloor n^2/4 \rfloor$ for all $n$. Indeed, let $G$ be a bipartite graph with parts $U$ and $V$ of sizes $\lfloor n/2 \rfloor$ and $\lceil n/2 \rceil$ respectively. Let $(u,v) \in E(G)$ if and only if $u \in U$ and $v \in V$. Then $e(G)=\lfloor n^2/4 \rfloor$ and $G$ is irredundant, as it contains no directed path of length two. Note that $f(7)=12=\lfloor 7^2/4 \rfloor$. Also, Lemma \ref{main_lemma} gives $f(8) \leq 16$ and thus $f(8)=16=\lfloor 8^2/4 \rfloor$, with equality only being possible if our graph has no double edges and no triangles (using the proof of Lemma \ref{main_lemma}).

Let $n \geq 7$ and suppose $f(n)=\lfloor n^2/4 \rfloor$ and $f(n+1)=\lfloor (n+1)^2/4 \rfloor$. Then Lemma \ref{main_lemma} gives $f(n+2) \leq \max \{\lfloor (n+1)^2/4 \rfloor +2, \lfloor n^2/4 \rfloor +3, \lfloor (n+2)^2/4 \rfloor \}$. But this maximum is equal to $\lfloor (n+2)^2/4 \rfloor$ for any $n \geq 7$. Again, the proof of Lemma \ref{main_lemma} shows that equality is only possible if our graph has no double edges and no triangles.

For $n=2$, it is clear that equality holds if and only if our graph is a double tree. Also, it is clear that any double tree is irredundant, for all values of $n$. Suppose that, for irredundant graphs with at most $n-1$ vertices, where $n \leq 6$, equality holds precisely for double trees. Let $G$ be an irredundant graph on $n$ vertices with $e(G)=f(n)=2n-2$. Using the proof of Lemma \ref{main_lemma}, we deduce that $G$ must contain some double edge, say $uv$. Construct $G'$ as in proof of Lemma \ref{main_lemma}, by contracting $uv$ into a single vertex $x$. Then $e(G')=2n-4=f(n-1)$ and thus $G'$ is a double tree. Let $y \in V(G')$ be any vertex other than $x$. If $yx$ is a double edge in $G'$, then either $yu$ or $yv$ was a double edge in $G$, as any other possibility would contradict $G$ being irredundant. Therefore, every edge in $G$ is double. Also, since $G$ is irredundant, it cannot contain any cycles. Therefore, $G$ is a double tree.

The above also shows that any maximum irredundant graph on seven vertices, which contains some double edge, has to be a double tree.

For $n>7$, the proof of Lemma \ref{main_lemma} shows that all the maximum irredundant graphs must contain no double edges and no triangles. Therefore, they must be the complete bipartite graphs given by Mantel's theorem. The same is also true for graphs on seven vertices that contain no double edges. Combining this with Lemma \ref{lemma_simple} shows that these graphs must be simple.
\end{proof}

\section*{Acknowledgements}
The author would like to thank his PhD supervisor, David Ellis, for his guidance and very useful comments, as well as Edward Crane and Erin Russell for suggesting this problem.

\bibliographystyle{plain}
\bibliography{references.bib}

\end{document}